\documentclass[12pt,english,times]{elsarticle}
\usepackage{color}
\usepackage{amsthm}
\usepackage{amsmath}
\usepackage{amssymb}
\usepackage{graphicx}
\makeatletter

\theoremstyle{plain}

\theoremstyle{plain}

\theoremstyle{plain}

\ifx\proof\undefined
\newenvironment{proof}[1][\protect\proofname]{\par
\normalfont\topsep6\p@\@plus6\p@\relax
\trivlist
\itemindent\parindent
\item[\hskip\labelsep
\scshape
#1]\ignorespaces
}{%
\endtrivlist\@endpefalse
}
\providecommand{\proofname}{Proof}
\fi
\theoremstyle{plain}

\usepackage{babel}

\providecommand{\lemmaname}{Lemma}
\providecommand{\propositionname}{Proposition}
\providecommand{\theoremname}{Theorem}

\newtheorem{proposition}{Proposition}[section]
\newtheorem{theorem}[proposition]{Theorem}

\newtheorem{lemma}[proposition]{Lemma}
\newtheorem{definition}[proposition]{Definition} 
 
\newtheorem{example}[proposition]{Example}

\usepackage{a4wide}
\usepackage[utf8x]{inputenc}
\usepackage{amsmath}
\usepackage{amsfonts}
\usepackage{amssymb}
\usepackage{enumerate}

\usepackage{graphicx}        %
\newcommand{\R}{\mathbb{R}}
\newcommand{\N}{\mathbb{N}}

\newcommand{\im}{{\rm im}\,}
\renewcommand{\ker}{{\rm ker}\,}

\newcommand{\rank}{{\rm rank}\,}

\newcommand{\dom}{{\rm dom}\,}

\newcommand{\pPi}{\mathnormal{\Pi}} %

\begin{document}

\title{Questions concerning  differential-algebraic operators:\\
Toward a reliable direct numerical treatment of differential-algebraic equations }

\author[KTH]{Michael Hanke}

\ead{hanke@nada.kth.se}

\address[KTH]{KTH Royal Institute of Technology, School of Engineering Sciences, 10044 Stockholm, Sweden}

\author[HUB]{Roswitha M{\"a}rz}

\ead{maerz@math.hu-berlin.de}

\address[HUB]{Humboldt University of Berlin, Institute of Mathematics, D-10099
Berlin, Germany}

\begin{abstract}
The nature of so-called differential-algebraic operators and their approximations is constitutive for the direct treatment of higher-index differential-algebraic equations. We treat first-order differential-algebraic operators in detail and contribute to justify the overdetermined polynomial collocation applied to higher-index differential-algebraic equations. Besides, we discuss several practical aspects concerning higher-order differential-algebraic operators and the associated equations.
\end{abstract}

\maketitle
\section{Introduction}\label{s.Introduction}
To a large extend, in the framework of numerical analysis, differential-algebraic equations (DAEs), in particular higher-index ones, are recognized as special ordinary differential equations (ODEs), and, accordingly, they are treated by means of derivative-array systems and an involved or preceded expensive index-reduction.  

In contrast, the experiments and theoretical contributions reportet in \cite{AscherPetzold1992,HMTWW,HMT} give rise to the  conjecture that next to the existing derivative-array based methods there is further potential toward a reliable direct numerical treatment of DAEs. The main aim of this note is to fill the gap between the theoretical convergence results for least-squares collocation methods \cite{HMTWW,HMT} and its practical realization. Moreover, we will explore the relevant scope concerning higher-order DAEs.
 
Recap well-known facts concerning first-order ordinary differential operators: Let $B(t)\in\R^{m,m}$ be continuous. The initial value problem (IVP)
\begin{align*}
 x'(t)+B(t)x(t)=g(t),\quad t\in [a,b], \quad x(a)=d\in \R^{m},
\end{align*}
has the unique solution
\begin{align*}
 x(t)=X(t,a)d+(Vg)(t),\; t\in [a,b], \quad d\in\R^m,\quad
 (Vg)(t)= \int_{a}^{t}X(t,s)g(s){\rm ds}.
\end{align*}
Introduce the operators $Tx:=x'+Bx$,
$T_{IC}x:=x(a)$,\, $\mathcal{T}x:=(Tx, T_{IC}x)$, such that the IVP is represented by the operator equation $\mathcal{T}x=(g,d)$ and the inverse of $\mathcal T$ is given by
\begin{align}\label{ODE}
 \mathcal{T}^{-1}(g,d)(t)=X(t,a)d+(Vg)(t),\; t\in [a,b].
\end{align}
As it is well known, these operators have useful properties making ODE problems easily accessible for the  numerical treatment: $ T:\mathcal C^1([a,b],\R^m)\rightarrow\mathcal  C([a,b],\R^m)$ is bounded and surjective, $\dim\ker T = m$, and hence, $T$ is fredholm, and 
$ \mathcal{T}:\mathcal C^1([a,b],\R^m)\rightarrow\mathcal  C([a,b],\R^m)\times \R^{m}$ is bounded and bijective, thus a homeomorphism. The same properties persist in the Hilbert space setting $ T:H^1((a,b),\R^m)\rightarrow L^{2}((a,b),\R^m)$, $ \mathcal{T}:H^1((a,b),\R^m)\rightarrow L^{2}((a,b),\R^m)\times \R^{m}$.
Such properties will survive for index one differential-algebraic equations in an appropriately modified
version. The situation in the higher index case is much more involved.
\medskip

The paper is organized as follows: Section~\ref{s.Example} deals with a small first-order index-three example and foreshadows the potential of overdetermined polynomial collocation. Section~\ref{s.Order1} is devoted to features of regular first-order arbitrary-index differential-algebraic operators acting in reasonable Hilbert spaces. We turn back to the overdetermined polynomial collocation in Section~\ref{s.OPC}, now considered by operators representing finite-dimensional approximations and provide with the main Theorem~\ref{t.main} sufficient conditions justifying overdetermined collocation.
Section~\ref{s.H} surveys higher-order differential-algebraic operators in this context.

Below, though using different norms we  mark the related norms by extra tags merely on those places where  confusions are actually imminent.

\section*{List of some symbols and abbreviations}
\begin{tabular}{ll}
$\R$, $\N$& sets of real and natural numbers\\
$\mathcal L(X,Y)$& for linear spaces $X,Y$ the space of linear operators\\
$\mathcal C(\mathcal I, X)$& space of continuous functions mapping $\mathcal I$ into $X$\\
$\mathcal C^k(\mathcal I, X)$& space of $k$-times continuously differentiable functions mapping $\mathcal I$ into $X$\\
$L^{2}:=L^{2}((a,b), \R^{m})$& Lebesque space of  functions mapping $(a,b)$ into $\R^{m}$\\
$H^{k}:=H^{k}((a,b), \R^{m})$& $:= W^{k,\,2}((a,b), \R^{m})$, Sobolev space of functions mapping $(a,b)$ into $\R^{m}$\\
$H^{1}_{D}:=H^{1}_{D}((a,b), \R^{m})$& $:=\{x\in L^{2} :\,Dx\in H^1$\} \\
$K^*$& adjoint of $K$\\
$K^-$& generalized inverse of $K$, $KK^-K=K$, $K^-KK^-=K^-$\\
$K^+$& Moore-Penrose inverse of $K$\\
$\ker K$& nullspace (kernel) of  $K$\\
$\im K$&  image (range) of $K$\\
$\langle \cdot,\cdot \rangle$& scalar product in $\R^m$\\
$(\cdot,\cdot)$& scalar product in function spaces\\
$|\cdot|$& Euclidian vector norm and spectral norm of matrices\\
$\|\cdot\|$& norm in function spaces, operator norm\\
$\oplus$& topological direct sum\\
DAE&differential-algebraic equation\\
ODE&ordinary differential equation\\
IVP, BVP&initial value problem, boundary value problem
\end{tabular} 

\section{A symptomatic example}\label{s.Example}
The DAE 
	\begin{align*}
		x'_{2}(t)+x_{1}(t) & =g_{1}(t),\\
		t\eta x'_{2}(t)+x'_{3}(t)+(\eta+1)x_{2}(t) & =g_{2}(t),\\
		t\eta x_{2}(t)+x_{3}(t) & =g_{3}(t),\quad t\in [0,1],
	\end{align*} 
has index 3 uniformly for every $\eta\in \R$. To each sufficiently smooth $y$ there exists ecactly one solution $x$.
The DAE is somewhat snaky, so that step by step integration methods generate waste unless an a priori or a posteriori incorporated  regularization via derivative array systems is incorporated, e.g., \cite[Example 8.5]{CRR}.

Here we set $\eta=-2$ and determine $g_1$, $g_2$, $g_3$ sucht that the solution becomes
	\begin{align*}
	x_{1}(t) =e^{-t}\sin t,\quad
	x_{2}(t) =e^{-2t}\sin t,\quad
	x_{3}(t) =e^{-t}\cos t.
	\end{align*}
We set $N\geq 1$ and approximate the solution components $x_{2}$ and $x_{3}$ by  continuously connected piecewise polynomials of degree $N$ and the component $x_{1}$ by possibly discontinuous piecewise polynomials of degree $N-1$ on uniform partitions of the interval $[0,1]$ with stepsize $h=1/n$. On each subinterval we choose $M$ collocation points. 

The classical or standard collocation procedures use $M=N$ collocation points per subinterval (cf.~\cite{LMW}). This results in $3nN$ equations to determine the $3nN+2$ parameters of the unknown collocation solution. In order to obtain a unique collocation solution, we pose additionally two consistent initial condition.
  
In contrast, choosing $M>N$ leads to an overdetermined collocation system which can be treated by a least-squares solver. Corresponding first experiments are reported in \cite{HMTWW} with $M=2N+1$. Table \ref{t.1} shows the componentwise maximal error in both versions: The standard collocation generates waste as expected, however, the overdetermined least-squares collocation  provides surprisingly nice results.
\begin{table}
 \caption{Componentwise maximal error, polynomial degree $N=3$}
\label{t.1}
\centering{}%
\begin{tabular}{c|ccc|ccc}
 $n$& \multicolumn{3}{c|}{Standard collocation} & \multicolumn{3}{c}{Least-squares collocation} \\
    &$i=1$ &$i=2$ & $i=3$ & $i=1$ & $i=2$ & $i=3$ \\
\hline
 20&5.56e+006 &3.03e+004 & 5.99e+004 & 2.09e-4 & 1.10e-06 & 2.18e-06 \\
 40&1.55e+017 &4.23e+014 & 8.41e+014 & 5.03e-5 & 1.31e-07 & 2.65e-07 \\
 80&5.70e+038 &7.76e+035 & 1.55e+036 & 1.23e-5 & 1.60e-08 & 3.20e-08 \\
160&3.36e+082 &2.29e+079 & 4.57e+079 & 3.06e-6 & 1.98e-09 & 4.00e-09 \\
320&4.93e+170 &1.68e+167 & 3.35e+167 & 7.68e-7 & 2.50e-10 & 5.00e-10
\end{tabular}
\end{table}
\medskip

Further experiments (cf.~\cite{HMTWW,HMT})  give rise to the conjecture that a much smaller number $M\geq N+1$ will do in general and that the special position of the collocation points does not matter. Tables \ref{t.2} and \ref{t.3} are quoted from \cite{HMT} to this effect.  At this place it should be noted that the overdetermined collocation method is treated in \cite{HMTWW,HMT} against the background of a least-squares problem in Hilbert spaces. For this reason, the errors are now measured  in $L^{2}$ and $H^{1}$ norms. The theoretically confirmed convergence order is $N-\mu+1=1$, but we observe order $2$. To date there is no theoretical recognition of this impressive, nice behavior!
\medskip

Table \ref{t.3} indicates the further interesting observation that even though no convergence is proved for $N=1$ the numerical approximations remain bounded and seem to converge with order 0.4.
 \begin{table}
	\caption{Error $(\rVert e_{1}\rVert^{2}_{L^{2}}+\rVert e_{2}\rVert^{2}_{H^{1}}+\rVert e_{3}\rVert^{2}_{H^{1}})^{\frac{1}{2}}$ of the collocation solution for $N=3$}
\label{t.2}
	\centering{}%
\begin{tabular}{r|c|c|c|c}
   & \multicolumn{2}{c|}{$M=2N+1$ uniform points} & \multicolumn{2}{c}{$M=N+1$ Gaussian points} \\
   \hline
   $n$ & error &order 
& error & order \\ 
\hline   
		10 & 6.31e-4 &      & 6.46e-4 &     \\
		20 & 1.44e-4 & 2.1  & 1.45e-4 & 2.2 \\
		40 & 3.47e-5 & 2.1  & 3.47e-5 & 2.1 \\
		80 & 8.53e-6 & 2.0  & 8.53e-6 & 2.0 \\
		160 &2.12e-6 & 2.0   &2.12e-6 & 2.0 \\
		320 & 5.27e-7 & 2.0  &5.27e-7 & 2.0   
	\end{tabular}
\end{table}
 \begin{table}
	\caption{Error $(\rVert e_{1}\rVert^{2}_{L^{2}}+\rVert e_{2}\rVert^{2}_{H^{1}}+\rVert e_{3}\rVert^{2}_{H^{1}})^{\frac{1}{2}}$  of the collocation solution for $N=1$}
\label{t.3}
	\centering{}%
\begin{tabular}{r|c|c|c|c}
   & \multicolumn{2}{c|}{$M=3$ uniform points} & \multicolumn{2}{c}{$M=2$ Gaussian points} \\
   \hline
$n$ & error & order & error & order\\ \hline
		10 & 5.65e-1 &         & 5.65e-1 &        \\
		20 & 3.93e-1 & 0.5  & 3.93e-1 & 0.5  \\
		40 & 2.49e-1 & 0.6  & 2.49e-1 & 0.7  \\
		80 & 1.85e-1 & 0.4  & 1.85e-1 & 0.4  \\
		160 &1.42e-1 & 0.4  & 1.42e-1 & 0.4  \\
		320 & 1.12e-1 & 0.3  & 1.12e-1 & 0.3 
	\end{tabular}
\end{table}
\medskip

The bounded operator $T: \{x\in \mathcal C([0,1],\R^{3}):x_{2},x_{3}\in \mathcal C^{1}([0,1],\R) \}   \rightarrow \mathcal C([0,1],\R^{3})$ associated with our test DAE  reads in detail
\begin{align*}
 Tx&=\begin{bmatrix}
       x_{2}'+x_{1}\\t\eta x_{2}'+x_{3}'+(\eta+1)x_{2}\\t\eta x_{2}+x_{3}
      \end{bmatrix},\; x_{1}\in \mathcal C([0,1],\R),\; x_{2},x_{3}\in \mathcal C^{1}([0,1],\R),
\end{align*}
and one immediately checks that 
\begin{align*}
\ker T=\{0\},\quad \im T=\{g\in \mathcal C([0,1],\R^{3}): g_{3}\in \mathcal C^{1}([0,1],\R),\,g_{2}-g_{3}'\in \mathcal C^{1}([0,1],\R)\},
\end{align*}
further  $\mathcal T=T$, since $T$ itself is injective, and hence, no initial or boundary conditions are allowed. The inverse operator 
\begin{align*}
 T^{-1}g= \begin{bmatrix}
       g_{1}-(g_{2}-g_{3}')'\\g_{2}-g_{3}'\\g_{3}-t\eta(g_{2}-g_{3}')
      \end{bmatrix},\quad g\in \im T.
\end{align*}
is unbounded in this setting since $\im T$ is a nonclosed subset in $\mathcal C([0,1],\R^{3})$. This is a fundamental contrast to the case of regular ODEs. Now we do not have  closed range and fredholm properties, and the main ingredient of the inverse of $\mathcal T$ is not a nice Volterra operator but a higher-order differential operator. We have 
\begin{align*}
 T^{-1}g= &\begin{bmatrix}
       g_{1}-g'_{2}+g''_{3}\\g_{2}-g'_{3}\\g_{3}-t\eta(g_{2}-g'_{3})
      \end{bmatrix}=
  \begin{bmatrix}
   0&0&1\\0&0&0\\0&0&0
  \end{bmatrix}g''+  
  \begin{bmatrix}
   0&-1&0\\0&0&-1\\0&0&t\eta
  \end{bmatrix}g'+
  \begin{bmatrix}
   1&0&0\\0&1&0\\0&-t\eta&1
  \end{bmatrix}g,\\
      & g\in \{v\in \mathcal C([0,1],\R^{3}):v_{2}\in \mathcal C^{1}([0,1],\R),  v_{3}\in \mathcal C^{2}([0,1],\R) \}\subset\im T.
\end{align*}
Analogous situations arise also, e.g., in the settings 
\begin{align*}
 T&: \mathcal C^{1}([0,1],\R^{3})  \rightarrow \mathcal C([0,1],\R^{3}),\\
 T&: \{x\in L^{2}((0,1),\R^{3}):x_{2},x_{3}\in H^{1}([0,1],\R) \}   \rightarrow  L^{2}((0,1),\R^{3}),\\
 T&: H^{1}((0,1),\R^{3})  \rightarrow L^{2}((0,1),\R^{3}),
\end{align*}
with their natural norms, \cite{Mae2014}.
\section{Regular first-order DA operators in a Hilbert space setting}\label{s.Order1}
\subsection{Setting}\label{subs.Setting}
We begin this part by considering the operator, henceforth called a \emph{DA operator},
\begin{equation*}
\mathring{T} \in\mathcal L(X,Y), \quad \mathring{T}x:=Ex'+Fx,\quad x\in \dom\mathring{T}:=\mathcal C^1([a,b],\R^m)\subseteq X, \;
\end{equation*}
so that the operator equation $\mathring Tx=g$ represents the so-called standard form DAE $Ex'+Fx=g$.
The coefficient functions $E, F:[a,b]\rightarrow \mathcal L(\R^{m},\R^{m})$ are at least  continuous. The function spaces $X$ and $Y$ will be specified later.

We suppose a nontrivial leading coefficient $E$ the nullspace of which is a $\mathcal C^1$-subspace varying in $\R^m$. We are looking for a Hilbert-space setting with a bounded DA operator.
\medskip

First we try $X=Y= L^{2}((a,b),\R^{m})$, with the usual norms.
In this setting, $\mathring T$ is unbounded, but densely defined and closable, see \cite{Mae2014}.
The closure of $\mathring T$,\quad $T:\dom T\subseteq X\rightarrow Y$, 
   is densely defined and closed, but also unbounded. We apply the usual graph-norm approach:
The space $X_{T}:=\dom T$  equipped with graph-norm,
\[ \rVert x\lVert_{T}=\rVert x\lVert+\rVert Tx\lVert,\; x\in X_{T}:=\dom T,\]
  is complete and $T:X_{T}\rightarrow Y$ is a bounded operator. 
 How can we specify $\dom T$ and $T$? To answer this question we need so-called proper factorizations of $E$.
\begin{definition}\label{d.properfac}
 The factorization $E=AD$ is called proper, if $0<k\leq m$,\; $A:[a,b]\rightarrow \mathcal L(\R^{k},\R^{m})$ is continuous, $D:[a,b]\rightarrow \mathcal L(\R^{m},\R^{k})$ is continuously differentiable,  and 
 \begin{align*}
 \ker A(t)\oplus\im D(t)=\R^k,\;t\in [a,b].
\end{align*}
\end{definition}
There are many possible proper factorizations. We fix an arbitrary one $E=AD$, 
put $B=F-AD'$, and observe that
\begin{align*}
 \mathring Tx=Ex'+Fx&=ADx'+Fx=A(Dx)'+(F-AD')x\\
 &=A(Dx)'+Bx,\quad \text{for all}\quad x\in\dom\mathring T.
\end{align*}
 Observe that $A(Dx)'+Bx=g$ is a DAE with so-called \textit{properly stated leading term}. This indicates how the closure $T$ as well as $\dom T$ look like.

\begin{theorem}\label{t.represent}
\begin{enumerate}[(i)]
 \item If $D:[a,b]\rightarrow \mathcal L(\R^{m},\R^{k})$ is continuously differentiable and has constant rank, then the function space 
\[
H^{1}_{D}((a,b),\R^{m})=\{x\in L^{2}((a,b),\R^{m}):Dx\in H^{1}((a,b),\R^{k})\}
\]
equipped with the inner product
\[
 (x,\bar x)_{H^{1}_{D}}=(x,\bar x)_{L^{2}}+((Dx)',(D\bar x)')_{L^{2}}
 \]
is a Hilbert space.
\item For each proper factorization $E=AD$, it results that
\begin{align*}
 \dom T=H^{1}_{D}((a,b), \R^{m}),\quad Tx=A(Dx)'+Bx, \; x\in \dom T.
 \end{align*}
 \item The norm $\lVert \cdot\lVert_{H^{1}_{D}}$ is
equivalent to the graph-norm $\rVert\cdot\lVert_{T}$.
\end{enumerate}
\end{theorem}
\begin{proof}
 (i) is an immediate  consequence of \cite[Lemma 6.9]{Mae2014}.
 
 (ii) The closure $T=\mathring T^{**}$ is provided by means of the biadjoint of $\mathring T$  in \cite[Theorem 3.1 (3)]{Mae2014}.
 
 (iii) Let $x\in \im T$ be given. From $A(Dx)'=Tx-Bx$ we obtain  $A^{-}A(Dx)'=A^{-}Tx-A^{-}Bx$. Owing to the proper factorization we may deduce the representation $(Dx)'=A^{-}Tx+ ((A^{-}A)'D-A^{-}B)x$. This yields
 \begin{align*}
  \lVert (Dx)'\rVert_{L^{2}}&\leq c_{1}(\lVert x\rVert_{L^{2}}+\lVert Tx\rVert_{L^{2}}),\\
  \lVert x\rVert_{H^{1}_{D}}^{2}&=\lVert x\rVert_{L^{2}}^{2}+\lVert (Dx)'\rVert_{L^{2}}^{2}\leq \lVert x\rVert_{L^{2}}^{2}+ c_{1}^{2} (\lVert x\rVert_{L^{2}}+\lVert Tx\rVert_{L^{2}})^2 \leq (1+c_{1}^{2})(\lVert x\rVert_{L^{2}}+\lVert Tx\rVert_{L^{2}})^2 .
 \end{align*}
On the other hand we have simply $\lVert x\rVert_{T}=\lVert x\rVert_{L^{2}}+\lVert Tx\rVert_{L^{2}}\leq (1+\lVert T\rVert)\lVert x\rVert_{H^{1}_{D}}$. 
\end{proof}
We emphasize that 
$\dom T$ and $T$ remain invariant under proper refactorizations $AD=\bar A\bar D$ since the closure of a densely defined closable operator is unique.\footnote{Note that this is at the same time consistent with \cite[Theorem 2.21]{CRR}.}
\subsection{Regular DA operators}\label{subs.Regularity}

We adopt the regularity notion in \cite[Sections 2.4.2 and 2.7]{CRR}. Regularity of a linear DAE is solely a matter of its coefficients $E, F$ and $A,D,B$, respectively.
The associated coefficient pair $(E,F)$ is regular on $[a,b]$ 
exactly if one (equivalently each) associated proper triple $(A,D,B)$ resulting from a proper factorization is regular on $[a,b]$. This is in full  accordance with the fact that the closure $T$ does not depend on the special factorization. 
 \begin{definition}\label{d.regularity}
 The DA operators $\mathring T$ and $T$ are regular with tractability index $\mu\in \N$ and characteristic values
 $0<r_{0}\leq\cdots\leq r_{\mu-1} < r_{\mu}=m$,
 if the associated DAE is regular on $[a,b]$ with these characteristics.
 
 A regular DA operator is said to be fine, if the coefficients $A,D,B$ are smooth enough for the existence of completely decoupling projectors. 
 \end{definition}
 Definition \ref{d.regularity} is consistent with \cite[Definition 2.1]{HMT}. 
 The operator $\mathring T$ and its closure $T$ share their index and characteristic values. Recall that 
$l:=m-\sum^{\mu-1}_{j=0}(m-r_j)\geq0$ is the dynamical degree of freedom of the DAE. 
$l=0$ may happen, see, e.g., the example DAE in Section~\ref{s.Example}.

A regular on $[a,b]$ DAE is associated with the so-called \textit{canonical projector function}
\begin{align*}
 \pPi_{can}\in \mathcal C([a,b],\mathcal L( \R^{m},\R^{m})),\quad \pPi_{can}(t)^{2}=\pPi_{can}(t),\;\rank\pPi_{can}(t)=l,\;  t\in[a,b],
\end{align*}
 which can be seen as generalization of the spectral projection for regular matrix pencils. Each IVP for the matrix-valued function $X$,
 \begin{align*}
  A(DX)'+BX=0,\quad X(c)=\pPi_{can}(c),\quad c\in[a,b],
 \end{align*}
is uniquely solvable. We refer to \cite[Section 2.4]{CRR}, \cite[Section 2.2]{LMW},\cite[Section 4.2]{Mae2014} for details, but note that, owing to the continuity of the coefficients, the solution of the above IVP, $X(\cdot,c)$ is well-defined and continuous with continuously differentiable part $DX(\cdot,c)$. $X(t,c)$ is called \textit{maximal-size fundamental solution matrix normalized at} $c$. It has constant rank $l<m$, and it may happen that $l=0$, see Section \ref{s.Example}. This is in contrast to the case of regular ODEs.
 
\medskip

In the following we focus on bounded DA operators given in their natural Hilbert spaces,
\begin{align*}
 T:H^{1}_{D}(\mathcal I,\R^{m})\rightarrow L^{2}(\mathcal I,\R^{m}),\quad
Tx:=A(Dx)'+Bx,\; x\in H^{1}_{D}(\mathcal I,\R^{m}),\quad \mathcal I:=(a,b).
 \end{align*}
If $T$ is regular, then
\begin{align*}
 \dim\ker T=l.
\end{align*}
Introduce the  operator $T_{IC}$ to capture initial conditions as well as the composed operator $\mathcal T$ by
\begin{align*}
T_{IC}:H^{1}_{D}(\mathcal I,\R^{m})\rightarrow \R^{l},\quad T_{IC}x=G_{a}x(a), \, x\in H^{1}_{D}(\mathcal I,\R^{m}),\\
 \mathcal T:H^{1}_{D}(\mathcal I,\R^{m})\rightarrow L^{2}(\mathcal I,\R^{m})\times\R^{l},\quad
 \mathcal T x=(Tx,T_{IC}\,x), \; x\in H^{1}_{D}(\mathcal I,\R^{m}).
\end{align*}
Thereby we suppose that $G_{a}\in\mathcal L(\R^{m},\R^{l})$ and $\ker D(a)\subseteq \ker G_{a}$. The latter condition ensures the relation $G_{a}=G_{a}D(a)^{+}D(a)$. Then owing to the continuous embedding $H^{1}\hookrightarrow \mathcal C$ the operator $T_{IC}$ is well-defined and bounded.
\medskip

Next we adopt the notion of accurately stated initial condition \cite[Definition 2.3]{LMW} accordingly.
\begin{definition}\label{d.IC}
The operator $T_{IC}$ is accurately stated if $\im \mathcal T=\im T\times \R^{l}$ and the composed operator $\mathcal T$ is injective.
\end{definition}
Owing to \cite[Corollary 2.2]{LMW},
if $T$ is regular, then $T_{IC}$ is accurately stated exactly if $\ker G_{a}\cap \im \pPi_{can}(a)=\{0\}$.
\medskip

Supposing $T$ to be fine and $T_{IV}$ to be accurately stated, the inverse of the composed operator $\mathcal T$ can be represented as (e.g.,\cite[Section 2.2]{LMW}, \cite[Section 2.6]{CRR}).

\begin{align}
 \mathcal T^{-1}(g,d)(t)&=\underbrace{X(t,a)G_{a}^{-}d+ (Vg)(t)}_{\in\; \im \pPi_{can}(t)} + \underbrace{(\mathfrak D g)(t)}_{\in\;\ker \pPi_{can}(t)},\quad (g,d)\in \im T\times\R^{l}, \label{DAE}\\
 &(Vg)(t)=\int_{a}^{t}X(t,s)G_{\mu}(s)^{-1}g(s){\rm ds},\label{DAEV}\\
 &\mathfrak Dg=v_{0}+v_{1}+\cdots+v_{\mu-1},\label{DAEN}
\end{align}
in which the functions $v_{i}$ are successively explicitly determined by simple multiplications with certain matrix coefficients, by differentiation of $Dv_j$, $j=i+1,\ldots,\mu-1$, and subsequent linear combinations,
\begin{align*}
 v_{\mu-1}&=\mathcal L_{\mu-1}g,\\
 v_{\mu-2}&=\mathcal L_{\mu-2}g -\mathcal N_{\mu-2,\,\mu-1}(Dv_{\mu-1})',\\
 \vdots\\
 v_{1}&=\mathcal L_{1}g -\sum_{s=2}^{\mu-1}\mathcal N_{1,\,s}(Dv_{s})'-\sum_{s=3}^{\mu-1}\mathcal M_{1,\,s}v_{s},\\
 v_{0}&=\mathcal L_{0}g -\sum_{s=1}^{\mu-1}\mathcal N_{0,\,s}(Dv_{s})'-\sum_{s=2}^{\mu-1}\mathcal M_{0,\,s}v_{s}.
\end{align*}
The coefficients $\mathcal N_{i,k},\mathcal M_{i,k},\mathcal L_{i}$, and $G_{\mu}$ are at least continuous. They are fully determined by $A,D,B$ via a sequence of admissible matrix functions corresponding to a complete decoupling.
\medskip

With the representation (\ref{DAE}) of the inverse $\mathcal T^{-1}$ we intend to emphasize, on the one hand, the partial resemblance to (\ref{ODE}). On the other hand, the second term  $\mathfrak Dg$ emerge for DA operators only. It is a differential operator, with may be higher order, see Section \ref{s.Example}. If $T$ has index $\mu>1$ then $\mathcal T^{-1}$ includes derivatives up to order $\mu-1$.

The representation (\ref{DAE})
 shows that the inverse of the composed operator actually decomposes into two parts. The first ``good'' part is close to (\ref{ODE}), i.e., the case of regular ODEs. This part may disappear. The second part is always present and representative for DA operators. Unfortunately, the canonical projector function which separates the parts is practically available in a few special cases only.

 We close this subsection by quoting further relevant results from \cite[Section 2]{HMT}
\begin{theorem}\label{t.images}
Let $T$ be fine with index $\mu\in \N$ and $T_{IC}$ be accurately stated.
\begin{enumerate}[(i)]
 \item If  $\mu=1$, then $\im T=L^{2}(\mathcal I,\R^{m})$, and the composed operator $\mathcal T$ is a homeomorphism.
 \item If $\mu>1$, then $\im T\subset L^{2}(\mathcal I,\R^{m}) $ is nonclosed and the inverse 
 $\mathcal T^{-1}$ as well as the Moore-Penrose inverse $T^{+}$ are no longer continuous. 
  Then the equation $\mathcal Tx=(g,d)$ is essentially ill-posed in Tikhonov's sense.
 \end{enumerate}
\end{theorem} 
\section{Justification of the overdetermind polynomial collocation}\label{s.OPC}
In this section we deal with regular higher-index DA operators $T$, the related composed operator $\mathcal T$ and their approximations $R_{\pi,M}TU_{\pi}$ and $\mathcal R_{\pi,M}\mathcal TU_{\pi}$. On the background of the corresponding properties we provide new sufficient convergence conditions for the overdetermined polynomial collocation.
\subsection{Basic technicalities}\label{subs.Basic}

We consider the linear IVP or BVP,
\begin{align}
 A(t)(Dx)'(t)+B(t)x(t)&=g(t),\quad t\in [a,b],\label{linDAE}\\  
  G_{a}x(a)+G_{b}x(b)&=d,\label{linBC}
\end{align}
with the constant matrix $D=[I\; 0]\in \mathcal L(\R^{m},\R^{k})$, $\rank D=k$, and at least continuous matrix coefficients
$A:[a,b]\rightarrow \mathcal L(\R^{k},\R^{m}),\; B:[a,b]\rightarrow \mathcal L(\R^{m},\R^{m})$.
The   DAE (\ref{linDAE}) is supposed to be fine in the sense of \cite[Section 2.6]{CRR} , with tractability index $\mu\in \N$ and dynamical degree of freedom $l\leq k$. Recall that $\mu>1$ necessarily implies $l<k$. The matrices $G_{a}, G_{b}\in \mathcal L(\R^{m},\R^{l})$ are supposed to satisfy the conditions
\begin{align}\label{G}
 \ker D\subseteq \ker G_{a},\quad \ker D\subseteq \ker G_{b}.
\end{align}
Condition (\ref{linBC}) is further supposed to be accurately stated  in the sense of \cite[Definition 2.3]{LMW}), so that the problems 
\begin{align}\label{homlinDAE}
 A(t)(Dx)'(t)+B(t)x(t)=0,\; t\in [a,b],\quad
  G_{a}x(a)+G_{b}x(b)=d,
\end{align}
are uniquely solvable for each $d\in \R^{l}$. In particular, the homogeneous linear BVP has the trivial solution only. 

The function $g$ is assumed to be admissible, so that the DAE \eqref{linDAE} is solvable. Then
the BVP \eqref{linDAE}, \eqref{linBC} has exactly one solution  $x_{*}$ to be approximated later on.

\medskip
Following the ideas of \cite{HMTWW,HMT} we represent the BVP (\ref{linDAE}), (\ref{linBC}) as operator equation $\mathcal Tx=y:=(g,d)$ in Hilbert spaces by introducing the spaces
\begin{align*}
 L^{2}:=L^{2}((a,b),\R^{m}),\quad H^{1}_{D}:=H^{1}_{D}((a,b),\R^{m}):=\{x\in L^{2}:Dx\in H^{1}((a,b),\R^{k})\},
\end{align*}
equipped with the inner products
\begin{align*}
 (x,\bar x)_{L^{2}}=\int^{b}_{a}\langle x(t),\bar x(t)\rangle {\rm dt},\; x,\bar x\in L^{2},\quad (x,\bar x)_{H^{1}_{D}}=(x,\bar x)_{L^{2}}+((Dx)',(D\bar x)')_{L^{2}((a,b),\R^{k})},\; x,\bar x\in H^{1}_{D}.
\end{align*}

and operators
\begin{align*}
 T&:H^{1}_{D}\rightarrow L^{2},\quad Tx=A(Dx)'+Bx,\; x\in H^{1}_{D},\\
 T_{BC}&:H^{1}_{D}\rightarrow \R^{l},
 \quad T_{BC} x=G_{a}x(a)+G_{b}x(b),\; x\in H^{1}_{D},\\
 \mathcal T&:H^{1}_{D}\rightarrow L^{2}\times\R^{l}=:Y,\quad \mathcal Tx=(Tx,T_{BC}x),\; x\in H^{1}_{D}.
\end{align*}
The operator $T_{BC}$ is well defined and bounded owing to condition (\ref{G}) and the continuous embedding $H^{1}((a,b),\R^{k})\hookrightarrow \mathcal C([a,b],\R^{k})$. Then, the DA operator $T$ as well as the composed operator $\mathcal T$ are obviously bounded. Moreover, $\mathcal T$ is injective and $\im \mathcal T=\im T\times \R^{l}$. At this place let us emphasize again that we focus our interest on higher-index DAEs, $\mu\geq 2$, but then $\im T$ is a nonclosed subset of $L^{2}$ and $\mathcal T^{-1}$ is an unbounded operator, cf. Section \ref{s.Order1}, also \cite{HMTWW,HMT,Mae2014}.
\bigskip

Given the partition
\begin{align}\label{partition}
 \pi: a=t_{0}<t_{1}<\cdots<t_{n}=b,
\end{align}
with stepsizes $h_{j}=t_{j}-t_{j-1}$, minimal stepsize $h_{\pi,\,min}$, and maximal stepsize $h_{\pi}$, we denote 
by $\mathcal C_{\pi}=\mathcal C_{\pi}([a,b],\R^{m})$ the space of piecewise continuous functions having breakpoints merely at the mesh points. Note that the supremum-norm $\lVert\cdot\rVert_{\infty}$ is well-defined for the elements of $\mathcal C_{\pi}$.

Next we fix a number $N\geq 1$ and introduce the space of ansatz functions to approximate $x_{*}$ by piecewise polynomial functions,
\begin{align}\label{ansatz}
 X_{\pi}=\{&x\in \mathcal C_{\pi}([a,b],\R^{m}):Dx\in \mathcal C([a,b],\R^{k}),\nonumber\\
 &x_{\kappa}\lvert_{[t_{j-1},t_{j})}\in \mathfrak P_{N},\, \kappa=1,\ldots,k,\quad
 x_{\kappa}\lvert_{[t_{j-1},t_{j})}\in \mathfrak P_{N-1},\, \kappa=k+1,\ldots,m,\; j=1,\ldots,n\}.
\end{align}
The finite-dimensional space $X_{\pi}$ is a closed subspace of $H^{1}_{D}$, and the latter decomposes into the topological sum $X_{\pi}\oplus X_{\pi}^{\bot}=H^{1}_{D}$. We agree upon that
\begin{align*}
 U_{\pi}:H^{1}_{D}\rightarrow H^{1}_{D} \quad\text{denotes the orthoprojection operator of }\; H^{1}_{D}\quad \text{onto }\; X_{\pi}.
\end{align*}
For later reference, the following norm in $X_\pi$ will be needed:
\begin{equation*}
 \lVert x \rVert_{C_D^1} = \lVert x \rVert_\infty + \lVert (Dx)' \rVert_\infty,\quad x\in X_\pi.
\end{equation*}
The ansatz space $X_{\pi}$ has dimension $nNm+k$. 
Choosing values 
\begin{align*}
 0<\tau_{1}<\cdots<\tau_{M}<1
\end{align*}
we specify $M\geq N+1$ collocation points per subinterval, i.e.,
\begin{align*}
 t_{ji}=t_{j-1}+\tau_{i}h_{j},\quad i=1,\ldots,M,\;j=1,\ldots,n,
\end{align*}
and are then confronted with the overdetermined collocation system of $nMm+l>nNm+k$ equations for providing an approximation $x\in X_{\pi}$, namely,
\begin{align}
 A(t_{ji})(Dx)'(t_{ji})+B(t_{ji}) x(t_{ji})-g(t_{ji})&=0,\quad i=1,\ldots,M,\;j=1,\ldots,n,\label{collDAE}\\
 G_{a}x(t_{0})+G_{b}x(t_{n})-d &=0.\label{collBC}
\end{align}
As a matter of course, the choice $M>N$ goes along with an overdetermined system (\ref{collDAE}),(\ref{collBC}) comprising more equations than unknowns. This is different from standard collocation methods for ODEs and index-1 DAEs, e.g.,~\cite{LMW}. Here we treat the overdetermined collocation system in a least-squares sense. 
\medskip
 
Let $R_{\pi,M}:\mathcal C_{\pi}([a,b],\R^{m})\rightarrow \mathcal C_{\pi}([a,b],\R^{m})$ denote the restriction operator which assigns to $w\in \mathcal C_{\pi}([a,b],\R^{m})$  the piecewise polynomial $R_{\pi,M}w\in \mathcal C_{\pi}([a,b],\R^{m})$ of degree less than or equal to $M-1$ such that the interpolation conditions,
\begin{align*}
 (R_{\pi,M}w)(t_{j i})=w(t_{j i}),\quad i=1,\cdots,M,\; j=1,\cdots, n,
\end{align*}
are satisfied.
We also assign to $w\in \mathcal C_{\pi}([a,b],\R^{m})$ the vector $W\in \R^{mMn}$,
\begin{align*}
 W=\begin{bmatrix}
    W_{1}\\\vdots\\W_{n}
   \end{bmatrix}\in \R^{mMn},\quad W_{j}=\left(\frac{h_{j}}{M}\right)^{\frac{1}{2}}
   \begin{bmatrix}
    w(t_{j 1})\\\vdots\\w(t_{j M})
   \end{bmatrix}\in \R^{mM},
\end{align*}
which yields (cf.~\cite[Subsection 2.3]{HMTWW})
\begin{align}\label{disc.norm}
 \rVert R_{\pi,M}w\lVert_{L^{2}}^{2}= W^{T}\mathcal L W, \quad w\in \mathcal C_{\pi}([a,b],\R^{m}),
\end{align}
with the matrix $\mathcal L$ being positive definite, symmetric and independent\footnote{The entries of $\mathcal L$ are fully determined by the corresponding $M$ Lagrangian basis polynomials, thus, by $M$ and $\tau_{1},\ldots, \tau_{M}$.}  of $h_{\pi}$. There are further constants $\kappa_{l},\kappa_{u}>0$ such that
\begin{align}\label{norm.equ}
 \kappa_{l}\;\lvert W\rvert^{2}\leq W^{T}\mathcal L W\leq \kappa_{u}\;\lvert W\rvert^{2},\quad W\in \R^{mMn}.
\end{align}
If $w\in \mathcal C_{\pi}([a,b],\R^{m})$ is of class $\mathcal C^{M}$ on each subinterval  of the partition $\pi$, then 
\begin{align}\label{rest}
 \lvert w(t)-(R_{\pi,M}w)(t)\rvert\leq\frac{1}{M!}\lVert w^{(M)}\rVert_{\infty}\; h_{\pi}^{M},\quad t\in[a,b].
\end{align}
Additionally, we introduce the restriction operator  $\mathcal R_{\pi,M}:\mathcal C_{\pi}([a,b],\R^{m})\times\R^{l}\rightarrow \mathcal C_{\pi}([a,b],\R^{m})\times\R^{l}$ by 
\[
 \mathcal R_{\pi,M}y =(R_{\pi,M}g, d),\quad y=(g,d)\in \mathcal C_{\pi}([a,b],\R^{m})\times\R^{l}.
\]
The overdetermined least-squares collocation means now that we seek an element $\tilde x_{\pi}\in X_{\pi}$
minimizing the functional 
\begin{align}
\psi_{\pi,M}(x)=\;\lVert \mathcal R_{\pi,M}(\mathcal Tx-y)\rVert^{2}_{L^{2}\times\R^{l}}= \lVert R_{\pi,M}(Tx-g)\rVert^{2}_{L^{2}}+\lvert T_{BC}x-d\rvert^{2},\quad x\in X_{\pi}\label{lscRF}.
\end{align}
With $w=Tx-g$ we may represent
\begin{align}
\psi_{\pi,M}(x)&=\;W^{T}\mathcal L \,W+\,\rvert T_{BC}x-d \lvert^{2},\quad x\in X_{\pi}, \label{lsc}
\end{align}
which reveals that by minimizing  $\psi_{\pi,M}(x)$ subject to $x\in X_{\pi}$ we actually provide a least-squares solution of the collocation system (\ref{collDAE}),(\ref{collBC}).
The mathematics behind is closely related to special properties of the restriction operator $R_{\pi,M}$ on the one hand, but on the other hand, to the  problem to minimize the functional
\begin{align}
 \psi(x)=\;\lVert \mathcal Tx-y\rVert^{2}_{L^{2}\times\R^{l}}=\;\rVert Tx-g\lVert^{2}_{L^{2}}+\rvert T_{BC}x-d\lvert^{2},\quad  x\in X_{\pi},\label{lscF}
\end{align}
for which (\ref{lscRF}) serves as approximation.

\subsection{The operators $R_{\pi,M} T U_{\pi}$ and $\mathcal R_{\pi,M}\mathcal T U_{\pi}$}
We begin this section by providing  useful norm inequalities. Regarding convergence properties for $h_{\pi}$ tending to zero we have in mind sequences of partitions. In favor for an easier reading we drop an extra labeling, but we thoroughly assure that the indicated constants  do not depend  on the partitions and stepsizes $h_{\pi}$ in fact. We allow partitions $\pi$ having quotients $h_{\pi}/h_{\pi,min}\leq r$, with a global bound $1\leq r<\infty$. A consequence of \cite[Theorem 3.2.6]{Ciarlet02} is that there exists a constant $c_K=c_K(r)$
such that 
\begin{equation}\label{invineq}
 \lVert z\rVert_{\infty}\leq c_{K} h_{\pi}^{-\frac{1}{2}}\;\lVert z\rVert_{L^{2}}
\end{equation}
for all functions $z\in \mathcal C_{\pi}([a,b],\R^{s})$ being a polynomial of degree less than or equal to $K$ on each subinterval of the partition $\pi$.

\begin{lemma}\label{l.norm}
There is a constant $\kappa>0$ such that
\begin{align*}
 \kappa h_{\pi}\, \lVert x\rVert^{2}_{\mathcal C^{1}_{D}}\leq \lVert x\rVert^{2}_{H^{1}_{D}}\leq (b-a)\,\lVert x\rVert^{2}_{\mathcal C^{1}_{D}},\quad x\in X_{\pi}.
\end{align*}
\end{lemma}
\begin{proof}
The inequality $\lVert x\rVert_{H^{1}_{D}}^{2}\leq (b-a) \lVert x\rVert_{\mathcal C^{1}_{D}}^{2}$ is evident for all $x\in X_{\pi}$. On the other hand, owing to \eqref{invineq} each arbitrary $x\in X_{\pi}$ satisfies
\begin{align*}
 \lVert x\rVert_{\infty}^{2}\leq c_{N}^{2} h_{\pi}^{-1} \lVert x\rVert_{L^{2}}^{2},\quad \lVert (Dx)'\rVert_{\infty}^{2}\leq c_{N-1}^{2} h_{\pi}^{-1} \lVert (Dx)'\rVert_{L^{2}}^{2},
\end{align*}
and finally
\begin{align*}
 \lVert x\rVert_{\mathcal C^{1}_{D}}^{2}\leq 2\lVert x\rVert_{\infty}^{2}+2\lVert (Dx)'\rVert_{\infty}^{2}\leq 2 (c_{N}^{2}+c_{N-1}^{2}) h_{\pi}^{-1} \lVert x\rVert^{2}_{H_{D}^{1}}=: \frac{1}{\kappa}h_{\pi}^{-1} \lVert x\rVert^{2}_{H_{D}^{1}}.
\end{align*}
\end{proof}
If $A$ and $B$ are constant matrices, then $Tx$ is piecewise polynomial with degree less than or equal to $N$ for $x\in X_{\pi}$. Owing to $M\geq N+1$ this leads to
\begin{align}\label{RTU1}
 R_{\pi,M}TU_{\pi}x=TU_{\pi}x,\quad \mathcal R_{\pi,M}\mathcal TU_{\pi}x=\mathcal TU_{\pi}x,\quad x\in H^{1}_{D}.
\end{align}
In general, the operators $R_{\pi,M}TU_{\pi}$ and thus $\mathcal R_{\pi,M}\mathcal TU_{\pi}$ are well-defined on $H^{1}_{D}$ since $TU_{\pi}x$ belongs to $\mathcal C_{\pi}$ for all $x\in H^{1}_{D}$.
\begin{proposition}\label{p.RTU}
Let the DA operator $T$ be fine with index $\mu\in \N$ and $T_{BC}$ be accurately stated. Let $M\geq N+1$ and let the entries of $A$ and $B$ be of class $\mathcal C^{M}$. Then the following assertions are valid for all sufficiently fine partitions $\pi$:
\begin{enumerate}[(i)]
 \item 
 There is a constant $C_{AB1}$ such that
 \begin{align}
  \lVert R_{\pi,M}TU_{\pi}x\rVert_{L^{2}}&\leq C_{AB1}\,\lVert x\rVert_{H^{1}_{D}},\quad x\in H^{1}_{D},\label{inequ.1}\\
  \lVert \mathcal R_{\pi,M}\mathcal TU_{\pi}x\rVert_{L^{2}\times\R^{l}}&\leq (C_{AB1}^{2}+\lVert T_{BC}\rVert^{2})^{\frac{1}{2}}\,\lVert x\rVert_{H^{1}_{D}},\quad x\in H^{1}_{D}.\nonumber
 \end{align}
 \item
 There is a constant $C_{AB2}$ such that
 \begin{align*}
  \lVert R_{\pi,M}TU_{\pi}x-TU_{\pi}x\rVert_{L^{2}}&\leq C_{AB2}\,h_{\pi}^{M-N-\frac{1}{2}}\,\lVert x\rVert_{H^{1}_{D}},\quad x\in H^{1}_{D},\\
  \lVert \mathcal R_{\pi,M}\mathcal TU_{\pi}x-\mathcal TU_{\pi}x\rVert_{L^{2}\times\R^{l}}&\leq  C_{AB2}\,h_{\pi}^{M-N-\frac{1}{2}}\,\lVert x\rVert_{H^{1}_{D}},\quad x\in H^{1}_{D}.
 \end{align*}
 \item If additionally $M\geq N+\mu$, then there is a constant $C$ such that
 \begin{align}\label{RTU2}
  \ker  \mathcal R_{\pi,M}\mathcal TU_{\pi}=\ker \mathcal TU_{\pi}=\ker U_{\pi},\quad \text{and}\quad
  \lVert \,(\mathcal R_{\pi,M}\mathcal TU_{\pi})^{+} \rVert \leq C h_{\pi}^{1-\mu}.
 \end{align}
  \item If the entries of $A$ and $B$ are polynomials of degree less than or equal to $N_{AB}$ and $M\geq N+1+N_{AB}$, then (\ref{RTU1}) and (\ref{RTU2}) remain valid, too. 
\end{enumerate}
\end{proposition}
\begin{proof}
 (i) We choose a number $K\geq 1$, $K+N\leq M$, $K$ interpolations nodes $0<\tilde\tau_{1}<\cdots<\tilde \tau_{K}<1$ and provide piecewise entrywise polynomial approximations $\tilde A$ and $\tilde B$ of degree $K-1$ such that
 \begin{align*}
  \tilde A(t_{j-1}+\tilde \tau_{i}h_{j})=A(t_{j-1}+\tilde \tau_{i}h_{j}),\;\tilde B(t_{j-1}+\tilde \tau_{i}h_{j})=B(t_{j-1}+\tilde \tau_{i}h_{j}),\quad i=1,\ldots,K,\;j=1,\cdots,n.
 \end{align*}
Then there are  constants $c_{1}, c_{2}$ such that
\begin{align*}
 &\lVert \tilde A-A\rVert_{\infty}\leq c_{1}\,h_{\pi}^{K}, \quad  \lVert \tilde B-B\rVert_{\infty}\leq c_{1}\,h_{\pi}^{K} \\ 
 &\lVert \tilde A\rVert_{\infty}\leq \lVert A\rVert_{\infty} +\lVert \tilde A-A\rVert_{\infty}\leq \lVert A\rVert_{\infty}+c_{1}h_{\pi}^{K}\leq c_{2},\; \lVert \tilde B\rVert_{\infty}\leq \lVert B\rVert_{\infty} +\lVert \tilde B-B\rVert_{\infty}\leq \lVert B\rVert_{\infty}+c_{1}h_{\pi}^{K}\leq c_{2}.
\end{align*}
The auxiliary operator $\tilde T:H^{1}_{D}\rightarrow L^{2}$
\begin{align*}
 \tilde Tx=\tilde A (Dx)'+\tilde B x,\quad x\in H^{1}_{D},
\end{align*}
is bounded, $\lVert \tilde T\rVert\leq \sqrt{2}c_{2}$ and owing to $M\geq N+K$ it holds that $R_{\pi,M}\tilde TU_{\pi}x=\tilde TU_{\pi}x$ for all $x\in H^{1}_{D}$.

Next, to each arbitrary $x\in X_{\pi}$ we form $v:=(\tilde T-T)x \in \mathcal C_{\pi}$ and the vector $V\in \R^{mMn}$, 
\begin{align*}
 V=\begin{bmatrix}
    V_{1}\\\vdots\\V_{n}
   \end{bmatrix}\in \R^{mMn},\quad V_{j}=\left(\frac{h_{j}}{M}\right)^{\frac{1}{2}}
   \begin{bmatrix}
    v(t_{j 1})\\\vdots\\v(t_{j M})
   \end{bmatrix}\in \R^{mM},
\end{align*}
which yields (cf.~(\ref{disc.norm}))
\begin{align*}
 \rVert R_{\pi,M}v\lVert_{L^{2}}^{2}= V^{T}\mathcal L V\leq \kappa_{u} \lvert V\rvert^{2}=\kappa_{u}\sum_{j=1}^{n}\frac{h_{j}}{M}\sum_{i=1}^{M}\lvert v(t_{ji})\rvert^{2}\leq \kappa_{u}(b-a)\lVert v\rVert_{\infty}^{2}.
\end{align*}
On the other hand, regarding Lemma \ref{l.norm} and $x=U_{\pi}x$ we estimate
\begin{align*}
 \lVert v\rVert_{\infty}\leq c_{1}h_{\pi}^{K} \lVert (Dx)'\rVert_{\infty}+c_{1}h_{\pi}^{K} \lVert x\rVert_{\infty}=c_{1}h_{\pi}^{K} \lVert x\rVert_{\mathcal C_{D}^{1}}
 =c_{1}h_{\pi}^{K} \lVert U_{\pi}x\rVert_{\mathcal C_{D}^{1}}\leq c_{1} h_{\pi}^{K-\frac{1}{2}}\frac{1}{\sqrt{\kappa}}
 \lVert x\rVert_{H_{D}^{1}}.
\end{align*}
For each arbitrary $x\in H^{1}_{D}$ it follows that
\begin{align*}
 \lVert R_{\pi,M}TU_{\pi}x\rVert_{L^{2}}&\leq \lVert R_{\pi,M}(\tilde T-T)U_{\pi}x\rVert_{L^{2}}+\lVert R_{\pi,M}\tilde TU_{\pi}x\rVert_{L^{2}}= \lVert R_{\pi,M}v\rVert_{L^{2}}+\lVert \tilde TU_{\pi}x\rVert_{L^{2}}\\
 &\leq \sqrt{\kappa_{u}(b-a)}\, c_{1}\frac{1}{\sqrt{\kappa}} h_{\pi}^{K-\frac{1}{2}}\lVert x\rVert_{H_{D}^{1}}+\sqrt{2}\,c_{2}\lVert x\rVert_{H_{D}^{1}},
\end{align*}
which verifies the inequality (\ref{inequ.1}) with a suitable bound $C_{AB1}$.
Then it also results that
\begin{align*}
 \lVert \mathcal R_{\pi,M}\mathcal TU_{\pi}x\rVert^{2}_{L^{2}\times\R^{l}}=\lVert  R_{\pi,M} TU_{\pi}x\rVert^{2}_{L^{2}}+\lvert  T_{BC}U_{\pi}x\rvert^{2}\leq (C_{AB1}^{2} +\lVert T_{BC}\rVert^{2})\lVert x\rVert^{2}_{H^{1}_{D}},
\end{align*}
 and the assertion is verified.
\medskip

 (ii) To each arbitrary $x\in X_{\pi}$ we set $w=A(Dx)'+Bx$ and derive on each subinterval of the partition $\pi$ that
 \begin{align*}
  w'&=A(Dx)''+A'(Dx)'+Bx'+B'x,\\
  \vdots\\
  w^{(M)}&=A(Dx)^{(M+1)}+\cdots+A^{(M)}(Dx)'+Bx^{(M)}+\cdots+B^{(M)}x.
 \end{align*}
Since $(Dx)^{(N+1)}$ and $x^{(N+1)}$ vanish identically, we obtain the inequality
\begin{align*}
 \lvert w^{(M)}(t)\rvert\leq c_{3} (\lVert x\rVert_{\mathcal C^{1}_{D}}+\lVert x'\rVert_{\mathcal C^{1}_{D}}+\cdots+ \lVert x^{(N)}\rVert_{\mathcal C^{1}_{D}}),
\end{align*}
with a constant $c_{3}$ being determined by the coefficients $A$ and $B$, and their involved  derivatives. By Lemma \ref{l.norm} this yields
\begin{align*}
 \lVert w^{(M)}\rVert_{\infty}\leq c_{3} \frac{1}{\sqrt{\kappa}} h_{\pi}^{-\frac{1}{2}} (\;\lVert x\rVert_{H^{1}_{D}}+\lVert x'\rVert_{H^{1}_{D}}+\cdots+ \lVert x^{(N)}\rVert_{H^{1}_{D}}).
\end{align*}
Owing to \cite[Lemma 4.2]{HMT} it follows that 
\begin{align*}
 \lVert w^{(M)}\rVert_{\infty}&\leq c_{3} \frac{1}{\sqrt{\kappa}} h_{\pi}^{-\frac{1}{2}} \left(\lVert x\rVert_{H^{1}_{D}}+ h_{\pi}^{-1}r\sqrt{C^{*}_{1}}\lVert x\rVert_{H^{1}_{D}}+\cdots+ h_{\pi}^{-N}r^{N}\sqrt{C^{*}_{N}} \lVert x\rVert_{H^{1}_{D}}\right).
\end{align*}
We emphasize that the $C^{*}_{i}$ are also constants independent of the partition and stepsize.
Then there is a constant $c_{4}$ such that
\begin{align*}
 \lVert w^{(M)}\rVert_{\infty}\leq c_{4} h_{\pi}^{-N-\frac{1}{2}}\lVert x\rVert_{H^{1}_{D}}.
\end{align*}
Finally we arrive at
\begin{align*}
 \lVert R_{\pi,M}TU_{\pi}x-TU_{\pi}x\rVert_{L^{2}}=\lVert R_{\pi,M}w-w\rVert_{L^{2}}\leq \frac{\sqrt{b-a}}{M!} h_{\pi}^{M}\lVert w^{(M)}\rVert_{\infty}\leq \frac{\sqrt{b-a}}{M!}c_{4} h_{\pi}^{M-N-\frac{1}{2}}\lVert x\rVert_{H^{1}_{D}}
\end{align*}
and further
\begin{align*}
 \lVert \mathcal R_{\pi,M}\mathcal TU_{\pi}x-\mathcal TU_{\pi}x\rVert_{L^{2}\times \R^{l}}=\lVert R_{\pi,M}TU_{\pi}x-TU_{\pi}x\rVert_{L^{2}}\leq \frac{\sqrt{b-a}}{M!}c_{4} h_{\pi}^{M-N-\frac{1}{2}}\lVert x\rVert_{H^{1}_{D}}.
\end{align*}
(iii)  The Moore-Penrose inverse of $\mathcal TU_{\pi}$ satisfies, by \cite[Theorem 4.1]{HMT}, the inequality
\begin{align*}
 \lVert (\mathcal TU_{\pi})^{+}\rVert\leq \frac{r^{\mu-1}}{c_{\gamma}}h_{\pi}^{-(\mu-1)},
\end{align*}
with a positive constant $c_{\gamma}$.

Denote by $\mathcal V_{\pi}$ and $\mathcal V_{\pi,M}$ the orthoprojectors of $L^{2}\times\R^{l}$ onto $\im \mathcal TU_{\pi}$ and $\im \mathcal R_{\pi,M} \mathcal TU_{\pi}$, respectively. Assertion (ii) implies now
\begin{align*}
\lVert \mathcal V_{\pi}\mathcal R_{\pi,M}\mathcal TU_{\pi}-\mathcal TU_{\pi}\rVert &=\lVert \mathcal V_{\pi}( \mathcal R_{\pi,M}\mathcal TU_{\pi}-\mathcal TU_{\pi})\rVert \\
&\leq \lVert  \mathcal R_{\pi,M}\mathcal TU_{\pi}-\mathcal TU_{\pi}\rVert\leq C_{AB2}h_{\pi}^{M-N-\frac{1}{2}},
\end{align*}
and further, for sufficiently fine partitions,
\begin{align*}
 \lVert (\mathcal TU_{\pi})^{+}\rVert\;\lVert \mathcal V_{\pi}\mathcal R_{\pi,M}\mathcal TU_{\pi}-\mathcal TU_{\pi}\rVert\leq C_{AB2}\frac{r^{\mu-1}}{c_{\gamma}}h_{\pi}^{M-N-\mu+\frac{1}{2}}\leq \frac{1}{2}.
\end{align*}
Next we represent
\begin{align*}
 \mathcal V_{\pi}\mathcal R_{\pi,M}\mathcal TU_{\pi}=\underbrace{\mathcal V_{\pi}(\mathcal R_{\pi,M}\mathcal TU_{\pi}-\mathcal TU_{\pi})}_{=:\mathfrak E}+\underbrace{\mathcal TU_{\pi}}_{=:\mathfrak A}.
\end{align*}
The subspace $\im \mathfrak A=\im \mathcal TU_{\pi}$ has finite dimension, $\im \mathfrak E\subseteq \im \mathfrak A$. Furthermore, it holds that  $\ker \mathfrak A=\ker \mathcal TU_{\pi}=\ker U_{\pi}$ and  $\ker \mathfrak E\supseteq \ker U_{\pi}$. By \cite[Lemma A.2]{HM1} it follows that
\begin{align*}
 \ker\mathcal V_{\pi}\mathcal R_{\pi,M}\mathcal TU_{\pi}=\ker \mathfrak A,\quad \lVert (\mathcal V_{\pi}\mathcal R_{\pi,M}\mathcal TU_{\pi})^{+}\rVert\leq \frac{\lVert \mathfrak A^{+}\rVert}{1-\lVert \mathfrak A^{+}\rVert \;\lVert \mathfrak E\rVert}\leq2 \lVert \mathfrak A^{+}\rVert.
\end{align*}
Now the inclusions
\begin{align*}
 \ker U_{\pi}\subseteq \ker \mathcal R_{\pi,M}\mathcal TU_{\pi}\subseteq \ker\mathcal V_{\pi}\mathcal R_{\pi,M}\mathcal TU_{\pi}=\ker U_{\pi}
\end{align*}
hold and, hence, $\ker\mathcal R_{\pi,M}\mathcal TU_{\pi}=\ker U_{\pi}$. In the end we compute
\begin{align*}
 (\mathcal R_{\pi,M}\mathcal TU_{\pi})^{+}&=(\mathcal R_{\pi,M}\mathcal TU_{\pi})^{+}\mathcal R_{\pi,M}\mathcal TU_{\pi}(\mathcal R_{\pi,M}\mathcal TU_{\pi})^{+}=U_{\pi}(\mathcal R_{\pi,M}\mathcal TU_{\pi})^{+}\\
 &=(\mathcal V_{\pi}\mathcal R_{\pi,M}\mathcal TU_{\pi})^{+}\mathcal V_{\pi}\mathcal R_{\pi,M}\mathcal TU_{\pi}\;(\mathcal R_{\pi,M}\mathcal TU_{\pi})^{+}=(\mathcal V_{\pi}\mathcal R_{\pi,M}\mathcal TU_{\pi})^{+}\mathcal V_{\pi}\mathcal V_{\pi,M},
\end{align*}
and hence $\lVert(\mathcal R_{\pi,M}\mathcal TU_{\pi})^{+}\rVert\leq 2\lVert \mathfrak A^{+}\rVert$.

(iv) This assertion is a direct consequence of the fact that $TU_{\pi}x$ is piecewise polynomial of degree less than or equal to $N+N_{AB}$.
\end{proof}

\subsection{Error estimations}
Recall that $x_{*}$ denotes the sought solution, i.e., $\mathcal Tx_{*}=y$ for given $y=(g,d)\in \im T\times \R^{l}$.
As descibed in Subsection \ref{subs.Basic} the overdetermined polynomial collocation actually means that we generate the  minimizer   of the functional $\phi_{\pi,M}$, cf. (\ref{lsc}), that is,
\begin{align*}
 \tilde x_{\pi}={\rm argmin}\{\, \phi_{\pi,M}(x):x\in X_{\pi}\}={\rm argmin}\{\, \lVert \mathcal R_{\pi,M}(\mathcal TU_{\pi}x-y)\rVert^{2}_{L^{2}\times\R^{l}}:x\in X_{\pi}\}=(\mathcal R_{\pi,M}\mathcal TU_{\pi})^{+}\mathcal R_{\pi,M}y,
\end{align*}
to approximate $x_{*}$. Now we provide a corresponding error estimation.

Suppose that the solution is smooth, $x_{*}\in \mathcal C^{N}([a,b],\R^m)$, $Dx_{*}\in \mathcal C^{N+1}([a,b],\R^m)$.
With the $N$ nodes $0<\tau_{*\,1}<\cdots<\tau_{*\,N}<1$,
the interpolating function $p_{*}\in X_{\pi}$ uniquely defined by
\begin{align*}
 p_{*}(t_{j-1}+\tau_{*\,i}h_{j})=x_{*}(t_{j-1}+\tau_{*\,i}h_{j}),\quad i=1,\ldots,N,\;j=1,\ldots, n,\quad Dp_{*}(t_{0})=Dx_{*}(t_{0}),
\end{align*}
satisfies the inequalities
\begin{align*}
 \lVert p_{*}-x_{*}\rVert_{\mathcal C^{1}_{D}}&\leq c_{*} h_{\pi}^{N},\\
  \lVert p_{*}-x_{*}\rVert_{H^{1}_{D}}&\leq  \sqrt{b-a}\;
c_{*} h_{\pi}^{N}=: c_{\alpha} h_{\pi}^{N},\\
\lVert U_{\pi}x_{*}-x_{*}\rVert_{H^{1}_{D}}&\leq
 \lVert p_{*}-x_{*}\rVert_{H^{1}_{D}}\leq  c_{\alpha} h_{\pi}^{N},
\end{align*}
in which the constant  $c_{*}$ is determined by $x_{*}$.
Next, owing to Proposition~\ref{p.RTU}(i), we may estimate
\begin{align*}
 \lVert \mathcal R_{\pi,M}( \mathcal Tp_{*}- \mathcal TU_{\pi}x_{*})\rVert &= \lVert \mathcal R_{\pi,M}( \mathcal TU_{\pi}p_{*}- \mathcal TU_{\pi}x_{*})\rVert= \lVert \mathcal R_{\pi,M} \mathcal TU_{\pi}(p_{*}-x_{*})\rVert \\
 &\leq (C_{AB}^{2}+\lVert T_{BC}\rVert^{2})^{\frac{1}{2}} \lVert p_{*}-x_{*}\rVert_{H^{1}_{D}}
 \leq (C_{AB}^{2}+\lVert T_{BC}\rVert^{2})^{\frac{1}{2}} c_{\alpha} h_{\pi}^{N}.
\end{align*}
Denoting $w_{*}=T(x_{*}-p_{*})\in \mathcal C_{\pi}$ and using the Lagrange basis polynomials we further derive 
\begin{align*}
 \lVert R_{\pi,M}w_{*}\rVert_{\infty}\leq \max_{j=1,\ldots,n}\; \max_{t_{j-1}\leq t\leq t_{j}}\sum_{i=1}^{M}\lvert l_{ji}(t)\rvert \;\lVert w_{*} \rVert_{\infty}=\max_{0\leq \tau\leq 1}\sum_{i=1}^{M} \prod_{s=1,s\neq i}^{M} \left\lvert\frac{\tau-\tau_{s}}{\tau_{i}-\tau_{s}}\right\rvert \;\lVert w_{*} \rVert_{\infty}=:C_{L}\lVert w_{*} \rVert_{\infty}.
\end{align*}
Because of 
\begin{align*}
 \lVert w_{*}\rVert_{\infty}\leq \max \{\lVert A\rVert_{\infty}, \lVert B\rVert_{\infty}\}\;\lVert x_{*}-p_{*}\rVert_{\mathcal C_{D}^{1}}\leq
 \max \{\lVert A\rVert_{\infty}, \lVert B\rVert_{\infty}\}\; c_{*} h_{\pi}^{N}
\end{align*}
we obtain 
\begin{align*}
  \lVert R_{\pi,M}T(x_{*}-p_{*})\rVert_{\infty}=\lVert R_{\pi,M}w_{*}\rVert_{\infty}\leq C_{L}\max \{\lVert A\rVert_{\infty}, \lVert B\rVert_{\infty}\}\; c_{*} h_{\pi}^{N},
\end{align*}
such that
\begin{align*}
 \lVert \mathcal R_{\pi,M}\mathcal T(x_{*}-p_{*}) \rVert^{2}_{L^{2}\times\R^{l}}&=\lVert R_{\pi,M} T(x_{*}-p_{*} )\rVert_{L^{2}}^{2}+\lvert T_{BC}(x_{*}-p_{*})\rvert^{2}\\
 &\leq (b-a) C_{L}^{2} \max \{\lVert A\rVert_{\infty}^{2}, \lVert B\rVert_{\infty}^{2}\} c_{*}^{2} h_{\pi}^{2N}+\,\lVert T_{BC}\rVert c_{\alpha}^{2} h_{\pi}^{2N} =: C_{R} h_{\pi}^{2N}.
\end{align*}
Proposition~\ref{p.RTU}(iii) implies, for $M\geq N+\mu$, that $U_{\pi}=(\mathcal R_{\pi,M}\mathcal TU_{\pi})^{+}\mathcal R_{\pi,M}\mathcal TU_{\pi}$, which gives rise to the error  representation 
\begin{align*}
 \tilde x_{\pi}-U_{\pi}x_{*}&=(\mathcal R_{\pi,M}\mathcal TU_{\pi})^{+}\mathcal R_{\pi,M}y-(\mathcal R_{\pi,M}\mathcal TU_{\pi})^{+}\mathcal R_{\pi,M}\mathcal TU_{\pi}x_{*}\\
 &=(\mathcal R_{\pi,M}\mathcal TU_{\pi})^{+}\mathcal R_{\pi,M}\mathcal Tx_{*}-(\mathcal R_{\pi,M}\mathcal TU_{\pi})^{+}\mathcal R_{\pi,M}\mathcal TU_{\pi}x_{*}\\
 &=(\mathcal R_{\pi,M}\mathcal TU_{\pi})^{+}\{\mathcal R_{\pi,M}\mathcal T(x_{*}-p_{*})-\mathcal R_{\pi,M}\mathcal TU_{\pi}(x_{*}-p_{*})\},
\end{align*}
so that 
\begin{align*}
 \lVert \tilde x_{\pi}-U_{\pi}x_{*}\rVert_{H^{1}_{D}}&\leq \lVert (\mathcal R_{\pi,M}\mathcal TU_{\pi})^{+}\rVert \; (\,\lVert \mathcal R_{\pi,M}\mathcal T(x_{*}-p_{*}) \rVert+\lVert \mathcal R_{\pi,M}\mathcal TU_{\pi}(x_{*}-p_{*})\rVert\,)\\
 &\leq C\, h_{\pi}^{1-\mu} (\sqrt{C_{R}}h_{\pi}^{N}+ (C_{AB}^{2}+\lVert T_{BC}\rVert^{2})^{\frac{1}{2}} c_{\alpha} h_{\pi}^{N}) =: C_{*} h_{\pi}^{N-\mu+1}.
\end{align*}
At the end we arrive at 
\begin{align*}
  \lVert \tilde x_{\pi}-x_{*}\rVert_{H^{1}_{D}}\leq  \lVert \tilde x_{\pi}-U_{\pi}x_{*}\rVert_{H^{1}_{D}} +  \lVert x_{*}-U_{\pi}x_{*}\rVert_{H^{1}_{D}}\leq (C_{*}+c_{\alpha}) h_{\pi}^{N-\mu+1}.
\end{align*}
We summarize the result in the following theorem.
\begin{theorem}\label{t.main}
Let the DA operator $T$ be fine with index $\mu\in \N$ and $T_{BC}$ be accurately stated. Let $M\geq N+\mu$ and let the entries of $A$ and $B$ be of class $\mathcal C^{M}$. Let $g\in \im T$, and the solution $x_{*}=\mathcal T^{-1}(g,d)$ be of class $\mathcal C^{N}$ with $Dx_{*}$ of class $\mathcal C^{N+1}$.
Then there is a constant $\bar C$ such that the error estimation
\begin{align}\label{estimation}
  \lVert \tilde x_{\pi}-x_{*}\rVert_{H^{1}_{D}}\leq \bar C h_{\pi}^{N-\mu+1}
\end{align}
is valid  
for all sufficiently fine partitions $\pi$. 
\end{theorem}
In contrast to Theorem \ref{t.main}, the earlier error estimation from \cite[Theorem 3.1 (a)]{HMT} is given for the least-squares approximation $x_{\pi}$ (cf.~(\ref{lscF})),
\begin{align*}
 x_{\pi}={\rm argmin}\{\, \phi(x):x\in X_{\pi}\}={\rm argmin}\{\, \lVert \mathcal TU_{\pi}x-y\rVert^{2}_{L^{2}\times\R^{l}}:x\in X_{\pi}\}=(\mathcal TU_{\pi})^{+}y.
\end{align*}
Then, supposing the entries of $A$ and $B$ to be polynomials of degree less than or equal to $N_{AB}$ and letting $M\geq N+1+N_{AB}$ the estimation (\ref{estimation}) is derived by a different technique (\cite[Theorem 5.1(a)]{HMT}). In the present context of operator properties this means that then the operators $ R_{\pi,M} TU_{\pi}$ and $ TU_{\pi}$ coincide, see Proposition \ref{p.RTU}. 

So far we know only sufficient convergence and order conditions. The question concerning an appropriate or even optimal choice of $N\geq 1$ and $M\geq N+1$ remains open. Of course, smaller $N$ and $M$ are associated with less computational effort. So far, in experiments $M=N+1$ works quite well. In general,  the practical performance is much better than we can substantiate till now. Much further analysis is needed. 

\section{Higher-order DA operators}\label{s.H}
General linear order-$s$ DA operators have the form (e.g.,~\cite{Cist,Bulatov,Mehr})
\begin{align}\label{H1}
 \mathring Tx=E_{s}x^{(s)}+\cdots+E_{1}x'+E_{0}x,\quad x\in \dom \mathring T=\mathcal C^{s}([a,b],\R^{m}),
\end{align}
with at least continuous matrix-coefficients $E_{i}$ and a singular leading coefficient $E_{s}$.
So far the consolidated knowledge of higher-order DAEs and the related operators is rather poor. 
Naturally the class of higher-order DA operators is much more complex than the class of first-order ones, nevertheless the numerical treatment of the corresponding DAEs might be often easier than expected.
In particular, systems of ODEs of mixed order, which can be handled by traditional approved software packages such as  COLNEW, COLDAE and BVPSUITE (\cite{BaderAscher,AscherSpiteri,KKPSW}),  actually apply to higher-order DAEs directly. For instance, the simple system
\begin{align*}
 x_{1}''+x_{1}=g_{1},\\
 x_{2}' +x_{1}+x_{2}=g_{2},
\end{align*}
corresponds to the DA operator $\mathring T:\dom \mathring T=\mathcal C^{2}([a,b],\R^{2})\subset \mathcal C([a,b],\R^{2})\rightarrow\mathcal C([a,b],\R^{2})$
\begin{align*}
 \mathring Tx=\begin{bmatrix}
     1&0\\0&0
    \end{bmatrix}x''
+\begin{bmatrix}
     0&0\\0&1
    \end{bmatrix}x'+\begin{bmatrix}
     1&0\\1&1
    \end{bmatrix}x,\quad x\in \dom \mathring T,
\end{align*}
and its extension $T$, with  $\dom T=\{z\in \mathcal C([a,b],\R^{2}):z_{1}\in \mathcal C^{2}([a,b],\R),z_{2}\in\mathcal C^{1}([a,b],\R) \}$,
\begin{align*}
 Tx=\begin{bmatrix}
     1\\0
    \end{bmatrix}
   ( \begin{bmatrix}
     1&0
    \end{bmatrix}x)''
+\begin{bmatrix}
     0\\1
    \end{bmatrix}
    (\begin{bmatrix}
     0&1
    \end{bmatrix}
    x)'+\begin{bmatrix}
     1&0\\1&1
    \end{bmatrix}x,\; x\in \dom T.
\end{align*}
One has $\im T=\mathcal C([a,b],\R^{2})$ and $\dim\ker T=3$, and hence $T$ is fredholm, so that IVPs and  BVPs can be stated in a well-posed way.
It arises the question
which further DA operators allow a reliable direct treatment of the corresponding DAE. For the time being we are not able to present a general answer. Below, we survey certain related aspects.
\medskip

In the early paper \cite{AscherPetzold1992}   higher-order Hessenberg DAEs arising from higher-order  ODEs subject to constraints are introduced and analyzed with respect to their perturbation index. 
Written in the form \eqref{H1} the associated operators are
\begin{align}\label{H3}
 \mathring Tx=
 \begin{bmatrix}
  I&0\\0&0
 \end{bmatrix}x^{(s)}+
 \sum_{i=1}^{s-1}
\begin{bmatrix}
  E_{i,11}&0\\E_{i,21}&0
 \end{bmatrix}x^{(i)}+
 \begin{bmatrix}
  E_{0,11}& E_{0,12}\\ E_{0,21}& E_{0,22}
 \end{bmatrix}x,\quad x\in \dom \mathring T=\mathcal C^{s}([a,b],\R^{m})
\end{align}
and, with properly involved derivatives by the additional matrix  
$D= \begin{bmatrix}
     I&0 \end{bmatrix}$, $\rank D=m_{1}$,
\begin{align}\label{H2}
 Tx&=
 \begin{bmatrix}
  I\\0
 \end{bmatrix}(Dx)^{(s)}+
 \sum_{i=1}^{s-1}
\begin{bmatrix}
  E_{i,11}\\E_{i,21}
 \end{bmatrix}(Dx)^{(i)}+
 \begin{bmatrix}
  E_{0,11}& E_{0,12}\\ E_{0,21}& E_{0,22}
 \end{bmatrix}x,\\
 &\quad x\in \dom  T=\{x\in \mathcal C([a,b],\R^{m}):Dx\in \mathcal C^{s}([a,b],\R^{m_{1}})\}.\nonumber
\end{align}
We set $X_{T}:=\dom T$ and introduce the norm $\lVert x\rVert= \lVert x\rVert_{\infty}+ \lVert Dx\rVert_{H^{s}}$, $x\in X_{T}$, so that $T: X_{T}\rightarrow \mathcal C([a,b],\R^{m})$ is bounded.
\medskip

If $E_{0,22}$ remains nonsingular, then it is evident that 
\[\im T=\mathcal C([a,b],\R^{m}),\quad \dim\ker T= s\,m_{1},
\]
i.e., $T$ is also fredholm. Therefore the operator $T$ is marked as index-1 operator independently of $s\geq 1$. The corresponding IVPs and BVPs are well-posed and can be treated even by standard polynomial collocation. By introducing new variables $v_{i}=(Dx)^{i},\;i=1,\ldots,s-1$, one obtains an associated first-order formulation of the DAE in the variables $v_{1},\ldots,v_{s-1}, x$, which has the same index 1.
\medskip

If $E_{0,22}$ vanishes identically, but $E_{s-1,21}E_{0,12}$ remains nonsingular, then 
\[\im T=\{g\in \mathcal C([a,b],\R^{m_{1}+m_{2}}):E_{0,12}(E_{s-1,21}E_{0,12})^{-1}g_{2}\in \mathcal C^{1}([a,b],\R^{m_{1}}) \},
\]
and $T$ is marked as index-2 operator independently of $s\geq1$. Although $\im T$ is a nonclosed subset in $\mathcal C([a,b],\R^{m})$, utilizing the special problem structure and incorporating special projections, corresponding BVPs with accurately stated boundary conditions can be directly treated by COLDAE, see \cite{AscherPetzold1992}. Again, substituting the derivatives $(Dx)^{(i)}$ by new variables in the DAE leads to a first-order index-2 system which can be solved by overdetermined least-squares collocation, too.
\medskip

If $E_{0,22}$ as well as $E_{1,21},\ldots,E_{s-1,21} $ vanish identically, but $E_{0,21}E_{0,12}$ remains nonsingular, then the operator $T$ has index $s+1$, and
\[\im T=\{g\in \mathcal C([a,b],\R^{m_{1}+m_{2}}):E_{0,12}(E_{0,21}E_{0,12})^{-1}g_{2}\in \mathcal C^{s}([a,b],\R^{m_{1}}) \}.
\]
Also here, substituting the derivatives $(Dx)^{(i)}$ by new variables leads to an index-$(s+1)$ first-order DAE, and the latter can be treated by overdetermined least-squares collocation.
\bigskip

The index notion used in \cite{AscherPetzold1992} is the perturbation index $\mu\in \N$ of a suitable first-order formulation\footnote{Not surprisingly, the classical procedure of turning a higher-order ODE into a first-order system applied to a DAE increases the differentiation index and leads to different solvability results and smoothness requirements. For details and examples we refer to \cite{Mehr}.}. In the operator context  we consider the extension $T$ of $\mathring T:\dom \mathring T\subset \mathcal C([a,b],\R^{m})\rightarrow  \mathcal C([a,b],\R^{m})$, turn then to the bounded operator $T:X_{T}\rightarrow \mathcal C([a,b],\R^{m})$ and show that the elements $g$ of $\im T$ are involved therein  together with parts of derivatives up to order $\mu-1$, and $\mu$ is the smallest such number. 

In contrast to the standard form \eqref{H1}, the version with properly involved derivatives can be seen as source of a reasonable first-order formulation. We conjecture that this idea applies also to further classes of DAEs. To emphasize the capabilities of this idea we use \cite[Example]{Mehr} for a demonstration.

\begin{example}\label{e.Mehr}
We refactorize the leading term of the second order operator $\mathring T$ given by
\begin{align*}
 \mathring Tx=
 \begin{bmatrix}
  1&t+1\\t&t^{2}+t
 \end{bmatrix}x''+
\begin{bmatrix}
  0&2\\0&2t
 \end{bmatrix}x'+
 \begin{bmatrix}
  1&t\\1+t&t^{2}+t+1
 \end{bmatrix}x
\end{align*}
by 
\begin{align*}
E_{2}=
 \begin{bmatrix}
  1&t+1\\t&t^{2}+t
 \end{bmatrix}=
 \begin{bmatrix}
  1\\t
 \end{bmatrix}\begin{bmatrix}
  1&t+1
 \end{bmatrix}=:A_{2}D,\quad Dx''=(Dx)''-2D'x'.
\end{align*}
The resulting operator T reads
\begin{align*}
  Tx&=A_{2}(Dx)''+E_{0}x=
 \begin{bmatrix}
  1\\t
 \end{bmatrix}(Dx)''+
 \begin{bmatrix}
  1&t\\1+t&t^{2}+t+1
 \end{bmatrix}x,\\
 X_{T}&=\{x\in \mathcal C([a,b],\R^{2}):Dx\in \mathcal C^{2}([a,b],\R )\},\\
 \im T&=\{g\in \mathcal C([a,b],\R^{2}):[-t\,1]g\in \mathcal C^{2}([a,b],\R) \}.
\end{align*}
Introducing the new variable $v=(Dx)'=x_{1}'+(t+1)x_{2}'+x_{2}$ leads to a first-order formulation with index $\mu=3$.\footnote{In \cite{Mehr} merely an index-4 first-order formulation was obtained with $v=x_{1}'+(t+1)x_{2}'$. }
\end{example}
\bigskip

A completely different approach to operators \eqref{H1} and the corresponding DAEs is proposed in \cite{Cist} in the context of a ring of  operators acting on $\mathcal C^{\infty}([a,b],\R^m)$, among them arbitrary-order differential operators of the form \eqref{H1} with real-analytic coefficients\footnote{Note that the usual semi-norm family defining the topology of $\mathcal C^{\infty}$ is much too strong to measure practically relevant approximation errors, e.g., \cite[Subsection 2.4.2]{Mae2014}}. This serves as background of the following index notion by means of \textit{left regularizing operators}
\begin{definition}\label{LRO}
Each operator $Ly=\sum_{i=0}^{k}F_{i}y^{(i)}$, $y\in \mathcal C^{k}([a,b],\R^m)$, such that the superposition $L\circ\mathring T$ is a regular ODE of the same order as $\mathring T$, i.e.,
\begin{align*}
 (L\circ \mathring T)x=\sum_{i=0}^{s}\bar E_{i}x^{(i)},\; x\in \mathcal C^{s+k}([a,b],\R^m), \quad \bar E_{s}\; \text{nonsingular,}
\end{align*}
is called left regularizing operator of $\mathring T$. The smallest possible $k$ is said to be the index  of $\mathring T$.
\end{definition}
Here we denote the index in the sense of Definition \ref{LRO} by $\mu_{C}$. The construction of left regularizers is closely related to the evaluation of derivative arrays, see \cite{Cist}. For $s=1$, constructing a left regularizer is equivalent to providing a so-called \textit{completion} ODE (\cite{Campbell1985}) and $\mu_{C}$ equals the differentiation index. In turn, for regular first-order DAEs the differentiation index equals the perturbation index and also the tractability index as well. 

However, for $s>1$ things are completely  different, so that $\mu_{C}$ is no longer helpful in view of the practical treatment of higher-order DAEs.  The following two simple examples allow a first insight. Both examples have the form \eqref{H3} resp. \eqref{H2}. We compare the perturbation index $\mu$ applied in \cite{AscherPetzold1992} and $\mu_{C}$.
\begin{example}[$\mu=1$,\; $\mu_{C}=s$]\label{e.H1}
\begin{align*} 
 \mathring Tx=
 \begin{bmatrix}
  x_{1}^{(s)}+x_{1}\\x_{2}
 \end{bmatrix},\quad 
L=
\begin{bmatrix}
 y_{1}\\y_{2}^{(s)}
 \end{bmatrix},\quad 
 (L\circ \mathring T)x=
 \begin{bmatrix}
  x_{1}^{(s)}+x_{1}\\x_{2}^{(s)}
 \end{bmatrix},\
\end{align*} 
\end{example}
\begin{example}[$\mu=s+1$,\; $\mu_{C}=2s$]\label{e.H2}
\begin{align*}
 \mathring Tx=
 \begin{bmatrix}
  x_{1}^{(s)}+x_{3}\\x_{2}^{(s)}\\x_{1}
 \end{bmatrix},\quad 
L=
\begin{bmatrix}
 y_{3}^{(s)}\\y_{2}\\y_{1}^{(s)}-y_{3}^{(2s)}
 \end{bmatrix},\quad 
 (L\circ \mathring T)x=
 \begin{bmatrix}
  x_{1}^{(s)}\\x_{2}^{(s)}\\x_{3}^{(s)}
 \end{bmatrix},\
\end{align*}
\end{example}

We finish by mentioning that in \cite{Mehr, Wunderlich}  the given DAE is transformed via derivative arrays to a so-called \textit{strangeness-free} mixed-order system which then can be handled by standard methods. Under additional quite special conditions, such a system is provided by evaluating  involved  matrix polynomials in \cite{Bulatov}.

\section{Conclusions}
We have explored properties of regular first-order DA operators and their finite-dimensional counterparts associated with the polynomial overdetermined least-squares collocation and provided an new convergence result.

We notice substantial progess in view of the consolidation of the polynomial overdetermined least-squares collocation. Nevertheless, there are essential open questions, e.g., concerning the choice of $N$ and $M$. 

Furthermore we have surveyed corresponding results concerning higher-order operators and the direct treatment of higher-order DAEs.

\bibliographystyle{plain} 
\bibliography{hm}

\end{document}